\documentclass[a4paper,10pt]{amsart} 
\usepackage[utf8]{inputenc}
\usepackage[english]{babel}
\usepackage{mathrsfs}
\usepackage{amsthm}
\usepackage{amsfonts}
\usepackage{amsmath}
\usepackage{amssymb}
\usepackage{faktor}
\usepackage{mathtools}
\usepackage{bm}

\usepackage{tikz}
\usepackage{esint}
\usepackage{centernot}
\usepackage{verbatim}
\usepackage[a4paper,top=3.5cm,bottom=3.5cm,left=2.5cm,right=2.5cm]{geometry}

\usepackage{cite}
\usepackage{color}
\definecolor{citegreen}{rgb}{0,0.8,0}
\definecolor{refred}{rgb}{0.8,0,0}
\usepackage[colorlinks, citecolor=citegreen, linkcolor=refred]{hyperref} 

\newtheorem{theorem}{Theorem}
\newtheorem{lemma}[theorem]{Lemma}
\newtheorem{corollary}[theorem]{Corollary}

\theoremstyle{definition}

\newtheorem{definition}[theorem]{Definition}

\newcommand{\R}{\mathbb R}\newcommand{\M}{\mathcal M}\newcommand{\N}{\mathcal N}
\numberwithin{equation}{section} \numberwithin{theorem}{section}
\allowdisplaybreaks

\DeclareMathOperator{\Div}{div}

\newcommand{\e}{\varepsilon}
\newcommand{\set}[1]{\{ {#1} \}}
\newcommand{\norm}[1]{\| {#1} \|}

\title[Subcritical Sacks-Uhlenbeck]{Sacks-Uhlenbeck type regularity for subcritical generalized $p$-harmonic maps into Homogeneous targets}
\author{Gianmichele Di Matteo, Tobias Lamm}
\address{Gianmichele Di Matteo, Scuola Superiore Meridionale, Largo San Marcellino 10, 80138 Napoli, Italy}
\address{Tobias Lamm, Karlsruhe Institute of Technology (KIT), Englerstrasse 2, 76131 Karlsruhe, Germany}

\begin{document}

\begin{abstract}
Adapting \cite{strz3}, we define generalized $p$-harmonic maps into Riemannian homogeneous targets, a notion of solutions not belonging to the energy space. Restricting our attention to the subcritical range $p$ greater than the domain dimension $n$, we show a uniform $C^{1,\alpha}$-regularity result for a sequence of such maps in the limit $p \searrow n$, assuming a uniform $n$-energy bound on its elements. The method of the proof follows the exact same lines as in \cite{strz3} but we need to check uniformity of estimates not previously considered there.
\end{abstract}

\maketitle

\begin{center}
\end{center}

\section{Introduction}
Given $p \in (1,+\infty)$, two Riemannian manifolds $(\M^n,g)$ and $(\N^k,h)$ and a map $u:\M \rightarrow \N$, we define its $p$-energy $D_p$ as 
\begin{equation*}
   D_p(u):= D_p(u;\M):=\tfrac{1}{p} \int_M |d u|^p d\mu_g, 
\end{equation*}
where $d\mu_g$ is the volume element of $g$, and $|d u|$ is the norm of $d u$ seen as a section of $T^*\M \otimes u^* T^*\N$. Critical points of the energy $D_p$ are distinguished in different classes: a critical point with respect to outer variations is called a \textit{weakly $p$-harmonic map}; if a map is critical also with respect to inner variations, it is called a \textit{stationary $p$-harmonic map}; finally, a map $u$ is called a \textit{locally minimizing $p$-harmonic map} if for any compact set $K\subset \subset \M$, $D_p(u)\le D_p(v)$ for all maps $v:\M \rightarrow \N$ coinciding with $u$ outside $K$. The Euler-Lagrange system of weakly $p$-harmonic maps is the following
\begin{equation}\label{eq.weaklypharmonic}
-\Div^g [|\nabla^g u|^{p-2} \nabla^g u]=|\nabla^g u|^{p-2} A_u(\nabla^g u, \nabla^g u),
\end{equation}
where $A_u$ is the second fundamental form of the target $(\N,h) \subset (\R^K,g_{euc})$, embedded isometrically via Nash's theorem.
The regularity of a $p$-harmonic map depends on its class; before listing some results, it is important to recall the definition of Riemannian homogeneous space. A \emph{closed Riemannian homogeneous space} $(\N,h)$ consists of a compact quotient of Lie groups $G/H$, where $G$ is connected and $H$ is a closed subgroup of it, endowed with a left invariant metric $h$.

Let us shortly summarize the literature regarding the regularity of $p$-harmonic maps. Firstly, in the case $p=2$, a partial regularity theory has been developed for locally minimizing maps \cite{sch0,sch1} and stationary harmonic maps \cite{cha0,bet,eva1,lin3,riv0}. Weakly harmonic maps are smooth when defined on $2$-dimensional surfaces \cite{hel2,hel3,riv1}, however they do not enjoy any partial regularity for higher dimensional domains \cite{riv3}.

For general $p$, the situation is more complicated and a full regularity theory is still not available. The partial regularity of maps locally minimizing $D_p$ was proven in \cite{har1,fuc0,luc1}; the regularity proven is $C^{1,\alpha}$, which is the best regularity to be expected in degenerate or singular elliptic problems (see for example \cite{iwa1}). For what regards stationary $p$-harmonic maps, the authors in \cite{mou0,strz0,tak0}, showed their partial regularity assuming that the target is a round sphere, all of them relying on a compensated compactness argument. A refinement of the same method, led Toro and Wang in \cite{tor} to extend this regularity to maps valued in some Riemannian homogeneous manifold $(\N,h)$. Finally, for what regards arbitrary targets $(\N,h)$, the conjectured partial regularity is still to be proven. Striking results in this direction, although partial, were obtained by Miskiewicz-Petraszczuk-Strzelecki in \cite{mis}, Riviere-Strzelecki in \cite{riv} and Martino-Schikorra in \cite{mar}, the latter two being the sharpest and most recent results. The central focus of this work is the study of the regularity properties of weakly $p$-harmonic maps in the subcritical regime $p>n$. As already remarked in our previous paper \cite{dim} Sobolev's embedding guarantees an initial a-priori H\"older continuity for these maps, which combined with the uniqueness result of Fardoun-Regbaoui \cite{far}, allows to deduce the locally minimizing property and hence a full $C^{1,\alpha}$-regularity on a small enough scale; however, this scale is degenerating to zero as $p\searrow n$, motivating the need of a more careful analysis. Such a regularity result originated in Sacks-Uhlenbeck's work \cite{sac}, Main Estimate $3.2$ there, where the authors address similar estimates for the so-called $\alpha$-harmonic maps. Furthermore, we are going to focus on an even weaker notion of solution, not belonging to the energy space, for which no Sobolev embedding can be applied.

The main advantage in assuming the target to be Riemannian homogeneous to the aim of proving such regularity results as in \cite{tor}, is that the Euler-Lagrange system of weakly $p$-harmonic maps into these targets can be equivalently expressed as a conservation law, as we shortly describe in the following.
\begin{definition}[Helein Fields]\label{def.helein}
Let $(G/H,h)$ be a Riemannian homogeneous space, $\mathfrak{g}$ denote the Lie algebra of $G$, and $\ell:=dim(\mathfrak{g})$ be its dimension. We call two $\ell$-tuples of smooth tangent vector fields $(X_1,...,X_\ell)$ and $(Y_1,...,Y_\ell)$ on $G/H$ \textit{Helein's fields} associated to $(G/H,h)$, provided the $X_a$'s are Killing with respect to $h$, and these vectors generate in the following sense: for any tangent vector field $V$ one can decompose it as follows
\begin{equation}\label{eq.helein}
V=\sum_{a=1}^\ell h(V,X_a)Y_a.
\end{equation}
\end{definition}
Helein in \cite{hel4}, Lemma $2.2$, proved that any Riemannian homogeneous manifold $(G/H,h)$ admits Helein's fields, in view of the transitivity of the Lie algebra action. They are not unique in general. Consider now a weakly $p$-harmonic map $u$, that is a solution of \eqref{eq.weaklypharmonic}. If $X$ is a smooth Killing tangent vector field on $(\N,h)$, then the vector field $W:=|\nabla^g u|^{p-2} h(\nabla^g u,X(u)) \in \Gamma(T \M)$ is $g$-divergence free in the distributional sense, see $(2.1)$ in \cite{tor}. They combine these identities to rewrite equation \eqref{eq.weaklypharmonic} as follows
\begin{equation}\label{eq.weakhomogeneous}
-\Div^g[|\nabla^g u|^{p-2} \nabla^g u^\alpha]=\sum_{a=1}^\ell g(W_a,\nabla^g [Y_a^\alpha(u)]) \quad \forall \alpha=1,...,K, \quad W_a:=|\nabla^g u|^{p-2} h(\nabla^g u,X_a(u)).
\end{equation}
Thanks to this new formulation, Toro and Wang in \cite{tor} used compensated compactness methods to prove the partial regularity claimed above, combining Hardy-BMO's duality \cite{fef} with classical result from \cite{coi0}.

In \cite{strz3}, Strzelecki noticed that, for spherical-targeted $p$-harmonic maps, distributional solutions to the formulation as conservation law
\begin{equation*}
\int_{\M} |\nabla^g u|^{p-2} (u^\alpha \nabla^g u^\beta  -u^\beta \nabla^g u^\alpha) \nabla^g \zeta =0, \quad \forall \zeta \in C^{\infty}_c(\M), \quad \forall \alpha,\beta=1,...,k+1,
\end{equation*}
can be defined assuming only the lower integrability $u \in W^{1,p-1}(\M;S^k)$. Similarly, we define
\begin{definition}
Suppose $(\N,h)$ is a Riemannian homogeneous space, and let $(X_1,...,X_\ell)$ and $(Y_1,...,Y_\ell)$ be some Helein's fields associated to it. We say that a map $u \in W^{1,p-1}(\M;\N)$ is a \textit{generalized $p$-harmonic map} if it satisfies for all $a=1,...,\ell$
\begin{equation}\label{eq.generalizedhomogeneous}
\int_{\M} |\nabla^g u|^{p-2} h(\nabla^g u,X_a(u)) \nabla^g \zeta =0, \quad \forall \zeta \in C^{\infty}_c(\M).
\end{equation}
\end{definition}
Notice that this definition is well posed as the $X_a$'s are smooth and bounded, hence the integral is well defined.

Theorem $1.5$ in \cite{strz3} shows $C^{1,\alpha}$-regularity of generalized $p$-harmonic maps into spheres under BMO-norm smallness and assuming integrability of the gradient $|\nabla^g u| \in \mathbb{L}^q$ with respect to an exponent $q \in (p-1,p)$ close enough to $p \le n$. We can now state our main theorem covering the complementary subcritical case:
\begin{theorem}\label{th.SUregularity}
There exist constants $P_0:=P_0(n,\M,g,\N,h) >n$ and $\alpha_0:=\alpha_0(n,\M,g,\N,h), \ \e_0:=\e_0(n,\M,g,\N,h) \in (0,1)$ and such that if $(u_p)_{(n,P_0)}\in W^{1,n}(B(x_0,4R_0);\N)$ is a family of generalized $p$-harmonic maps defined on a ball $B(x_0,4R_0) \subset \M$, such that for all $p \in (n,P_0)$ we have the uniform energy bound $D_n(u_p;B(x_0,4R_0)) \le \e_0^n$, then $u_p \in W_{loc}^{1,P_0}(B(x_0,4R_0);\N)$ solve \eqref{eq.weaklypharmonic}, and $u_p \in C_{loc}^{1,\alpha_0}(B(x_0,4R_0);\N)$ for all $p \in (n,P_0)$, along with locally uniform bounds.
\end{theorem}
Notice that in our case both the BMO-smallness and the integrability condition of Theorem $1.5$ in \cite{strz3} can be solely expressed in terms of the $\mathbb{L}^n$-integrability. As a consequence, globally defined generalized $p$-harmonic maps for $p \in (n,P_0)$ must be $C^{1,\alpha_0}$ on the entire domain manifold $\M$, see Corollary \ref{cor.global}.

The argument follows closely the one in \cite{strz3}, relying on Iwaniec-Sbordone's stability estimate for a non-linear Hodge decomposition \cite{iwa3,iwa2}.
The main novel difficulty is to keep track of the uniformity of the estimates for $p\in (n,P_0)$.

Let us explicitely remark that we cannot reach the value of the threshold $P_0=n+1$, which would be a generalization of Iwaniec-Sbordone conjecture in \cite{iwa2}. Indeed, Colombo-Tione in \cite{col}, see Theorem $1.2$ there, construct strikingly counterexamples to this conjecture with trivial target $\N=\mathbb{R}$, verifying even one-side bounds on directional derivatives. In short, they constructed on an open ball $\Omega \subset \R^2$ and for all $p \in (1,+\infty)$, distributional solutions to
\begin{equation*}
\Div[|\nabla u|^{p-2} \nabla u]=0, \quad \text{on } \Omega,
\end{equation*}
which belong to $W^{1,p-1+\e}(\Omega)$, for some $\e=\e(p)>0$, and they can even assume $u \in C^{0,\alpha}(\Bar{\Omega})$, which however do not belong to the energy space:
\begin{equation*}
\int_\Omega |\nabla u|^p=+\infty.
\end{equation*}
In conclusion, let us remark that the same method of proof allows to prove, for any fixed $\underline{P}>n$, the uniform regularity of generalized $p$-harmonic maps $u_p$ for $p$ in a short enough interval $(\underline{P},\overline{P})$, assuming they satisfy uniformly the smallness energy bound $ R_0^{\underline{P}-n} D_{\underline{P} }(u_p,B(x_0,2R_0)) \le \e_0^{\underline{P} }$, see Theorem \ref{th.highsubcritical}.
\subsection{Structure of the paper}
In Section \ref{sec.preliminary} we discuss the main preliminary material needed to prove the main Theorem. In Section \ref{sec.regularity} we prove the main regularity Theorem \ref{th.SUregularity}.
\subsection{Acknowledgements}
The first named author has been partially supported by the PRIN Project 2022AKNSE4 \emph{Variational and Analytical aspects of Geometric PDEs} during the writing of this paper.

\section{Preliminary}\label{sec.preliminary}
In this Section we gather the preliminary material we will need to prove our main Theorem \ref{th.SUregularity}. We follow closely \cite{strz3} paying careful attention to the uniformity of the estimates as $p \searrow n$.

To start, let us recall the Poincar\'e-Wirtinger inequality $(7.45)$ in \cite{gil} applied to the case of $S=\Omega=B(x_0,R_0)$: for all $q \in [1,+\infty)$ and for all $u \in W^{1,q}(B(x_0,R_0))$ we have
\begin{equation}\label{eq.poincare}
\norm{ u-[u]_{B(x_0,R_0)} }_q \le 2^n R_0 \norm{\nabla u}_q.
\end{equation}
\subsection{Nonlinear Hodge decomposition and Iwaniec-Sbordone's stability estimate}
The standard Hodge decomposition consists in splitting a given vector field $X$ as a sum of a gradient vector field $\nabla v$ and a divergence free vector field $H$. It is quite natural to ask whether the elements of this decomposition satisfy some bounds in suitable functional spaces. The easiest example is when $X \in \mathbb{L}^2$, in which case the norm bounds follow by the continuity of the projection into the subspace of gradient vector fields. The case $X \in \mathbb{L}^p$ was obtained by Iwaniec-Martin in \cite{iwa4}, then refined and applied to quasiconformal and quasiregular mappings by T. Iwaniec \cite{iwa3}, and to weak minima by Iwaniec-Sbordone \cite{iwa2}. The result of most interest for us is Theorem $3$ in \cite{iwa2}, a refinement of Theorem $8.1$ in \cite{iwa3}, which expresses quantitatively the following heuristics. For $\e \sim 0$ we can expect a vector field $X$ of the form $X=|\nabla w|^{\e} \nabla w$ to be close to a gradient, meaning that its Hodge divergence-free component should be small. The statement we are going to apply is a global version of Theorem $3$ in \cite{iwa2}.
\begin{theorem}[Theorem $2.1$ in \cite{strz3}]\label{th.hodge}
Let $r \in (1,+\infty)$, $w \in W^{1,r}(\R^n)$ and $\e \in (-1,r-1)$. Then there exist $v \in W^{1,\frac{r}{1+\e}}(\R^n)$ and $H \in \mathbb{L}^{\frac{r}{1+\e}}(\R^n;\R^n)$ such that $|\nabla w|^{\e} \nabla w=\nabla v+H$, and $H$ is divergence-free in the distributional sense. Moreover, the following stability estimate holds
\begin{equation}
\norm{H}_{\frac{r}{1+\e}} \le C(r) |\e| \norm{\nabla w}_r^{1+\e}.
\end{equation}
Finally, if $1<s_1<s_2<+\infty$ are two fixed numbers such that $r,\tfrac{r}{1+\e} \in (s_1,s_2)$, then
\begin{equation}
C(r) \le \tfrac{2 r (s_2-s_1)}{(r-s_1)(s_2-r)}(A(s_1)+A(s_2)),
\end{equation}
where $A(s)$ is the norm of the operator $Id+(\mathcal{R}_{ij}):\mathbb{L}^s(\R^n;\R^n) \longrightarrow \mathbb{L}^s(\R^n;\R^n)$, where $\mathcal{R}_{ij}=\mathcal{R}_{i} \circ \mathcal{R}_{j}$ the second order Riesz transform with respect to indices $i,j \in \set{1,...,n}$.
\end{theorem}
Later, we are going to apply this Theorem by choosing $r:=p \in (n,P_0)$, $\e:=\tfrac{p-n}{n}$, so that $n=\tfrac{r}{1+\e}$. Therefore, we can choose $s_1:=\tfrac{3}{2}$, $s_2:=n+2$, assuming $P_0 \le n+1$, and estimate the constants as
\begin{equation}\label{eq.boundc1}
C(p) \le \tfrac{2 p (n+2-\tfrac{3}{2})}{(p-\tfrac{3}{2})(n+2-p)}(A(\tfrac{3}{2})+A(n+2)) \le \tfrac{2 (n+1) (2n+1)}{2n-3} [2+\tfrac{\pi^2}{4} \tan (\tfrac{\pi}{3})^2+\tfrac{\pi^2}{4} \cot (\tfrac{\pi}{2(n+2)})^2] =: C_1(n),
\end{equation}
where we have argued as in the proof of Theorem $8.14$ in \cite{iwa4} to deduce
\begin{align}
A(\tfrac{3}{2}) \le 1+ \tfrac{\pi^2}{4} \tan (\tfrac{\pi}{3})^2,\quad 
A(n+2) \le 1+ \tfrac{\pi^2}{4} \cot (\tfrac{\pi}{2(n+2)})^2,
\end{align}
and used the bounds $n<p<P_0\le n+1$.
These estimates come from Theorem $1.1$ in \cite{iwa4}.
Combining $|\e| \le P_0-1 \le n$ and the Hodge decomposition we also deduce 
\begin{equation}\label{eq.gradienthodge}
\norm{\nabla v}_{\frac{r}{1+\e}} \le \norm{|\nabla w|^{\e} \nabla w}_{\frac{r}{1+\e}}+\norm{H}_{\frac{r}{1+\e}} \le  (C_1(n) n+1) \norm{\nabla w}_r^{1+\e}.
\end{equation}
\subsection{Hardy and BMO spaces}
Consider a non-negative bump function $0 \le \phi \in C^{\infty}_c(\R^n)$, with integral $\int \phi =1$, and set $\phi_\lambda (x):= \lambda^{-n} \phi (\tfrac{x}{\lambda})$. We say that a function $f \in \mathbb{L}^1$ is in the \textbf{Hardy space} $\mathcal{H}^1(\R^n)$ if
\begin{equation}
f_*(x):=\sup_{\lambda>0} |\phi_\lambda * f|(x) \in \mathbb{L}^1(\R^n).
\end{equation}
We can endow $\mathcal{H}^1(\R^n)$ with the norm
\begin{equation}
\norm{f}_{\mathcal{H}^1(\R^n)}:= \norm{f}_1+\norm{f_*}_1,
\end{equation}
in order to make it a Banach space.

Coifman-Lions-Meyer-Semmes in \cite{coi0} showed that, for any $q \in (1,+\infty)$, $b \in W^{1,q}(\R^n)$ and $E \in \mathbb{L}^{q'}(\R^n;\R^n)$, where $E$ is divergence free in the distributional sense, then $E \cdot \nabla b \in \mathcal{H}^1(\R^n)$ and the following estimate holds
\begin{equation}
\norm{E \cdot \nabla b}_{\mathcal{H}^1(\R^n)} \le C(n,q) \norm{\nabla b}_q \norm{E}_{q'}.
\end{equation}
This estimate depends crucially on $q$, it certainly degenerates as $q$ approaches $1$ and $+ \infty$, and from the proof in \cite{coi0}, a priori it may also diverge as $q \searrow n$ (although the inequality is true for all $q \ge n$), in view of the application of Sobolev's embedding in Lemma II.1 there. This motivates the need of the following Lemma:
\begin{lemma}[Uniform CMLS Lemma]\label{lemma.CMLSglobal}
There exists a constant $C_2(n)$ such that for all $p \in (n,n+\tfrac{1}{2})$, for any $b \in W^{1,\tfrac{n}{n-p+1}}(\R^n)$ and $E \in \mathbb{L}^{\tfrac{n}{p-1}}(\R^n;\R^n)$, where $E$ is divergence free in the distributional sense, then $E \cdot \nabla b \in \mathcal{H}^1(\R^n)$ and the following uniform estimate holds
\begin{equation}\label{eq.CMLS}
\norm{E \cdot \nabla b}_{\mathcal{H}^1(\R^n)} \le C_2(n) \norm{\nabla b}_{\tfrac{n}{n-p+1},\R^n} \norm{E}_{\tfrac{n}{p-1},\R^n}.
\end{equation}
\end{lemma}
\begin{proof}
The uniformity of the constant $C_2$ for $p \in (n,n+\tfrac{1}{2})$ can be deduced carefully analyzing the argument in \cite{coi0}, Section $2$. In Theorem $II.1$ there, we choose their parameter "p" to be $\tfrac{n}{p-1}$. Then we apply Lemma $II.1$ with parameters $\alpha:=\tfrac{n}{p-1}$ and $\beta:=\tfrac{n}{n-p+2}$. The proof of that Lemma uses Sobolev's embedding to bound the $\alpha'=\beta^*=\tfrac{n}{n-p+1}$-norm of $b$ minus its average with the $\beta=\tfrac{n}{n-p+2}$-norm of $\nabla b$. It is clear that this constant is independent of $p \in (n,n+\tfrac{1}{2})$ as $\beta \in (\tfrac{n}{2},\tfrac{2 n}{3})$ for these values. The other uniformity to be checked is the bound on the Hardy–Littlewood maximal function, which however is again clear as both $\alpha:=\tfrac{n}{p-1} \in (\tfrac{2n}{2n-1},\tfrac{n}{n-1})$ and $\beta:=\tfrac{n}{n-p+2} \in (\tfrac{n}{2},\tfrac{2 n}{3})$ are bounded away uniformly from $1$ and $+\infty$.
\end{proof}
Let us remark that the uniformity claimed in the analogous Theorem $2.2$ of \cite{strz3} is on intervals $[1+\delta_0,1+\delta_0^{-1}]$ \textit{compactly contained in} $(1,m)$, coherently to the case treated there.

As a Corollary, we can prove a local version of Lemma \ref{lemma.CMLSglobal}, inspired by the analogous Corollary $3$ \cite{strz2}, which allows us to completely bypass the concept of local Hardy space $\mathcal{H}^1_{loc}$.
\begin{corollary}[Uniform local CMLS Lemma]
Let $B$ be a ball in $\R^n$. Let $p \in (n,n+\tfrac{1}{2})$, $b \in W^{1,\frac{n}{n-p+1}}(B)$ and $E \in \mathbb{L}^{\frac{n}{p-1}}(B;\R^n)$, and suppose that $E$ is divergence free in the distributional sense. Then there exists a function $h \in \mathcal{H}^1(\R^n)$ such that
$h=\nabla b \cdot E$ in $B$, and that satisfies the bound
\begin{equation}\label{eq.CMLSlocal}
\norm{h}_{\mathcal{H}^1(\R^n)} \le C_3 \norm{\nabla b}_{ \frac{n}{n-p+1} } \norm{E}_{\frac{n}{p-1}}.
\end{equation}
The constant $C_3=C_3(n)$ is dimensional and does not depend on the size of $B$ and on the choice of $p \in (n,n+\tfrac{1}{2})$.
\end{corollary}
\begin{proof}
Following the proof of Corollary $3$ in \cite{strz2}, we need to check the uniformity for $p \in (n,n+\tfrac{1}{2})$ of the norm bounds on two operators: the extension operator $E_s:W^{1,s}(B;\Lambda^\ell(\R^n)) \rightarrow W^{1,s}_{loc}(\R^n;\Lambda^\ell(\R^n))$, defined on $\ell$-differential forms, for $s=\frac{n}{p-1}$ and $s=\frac{n}{n-p+1}$, which holds as both these exponents are bounded away uniformly from $1$ and $+\infty$; secondly, the homotopy operator $T:\mathbb{L}^{\frac{n}{p-1}}(B;\Lambda^\ell(\R^n)) \rightarrow W^{1,\frac{n}{p-1} }_{loc}(\R^n;\Lambda^{\ell-1}(\R^n))$ defined in $(4.5)$ of \cite{iwa5}, verifying for all forms $\omega \in W^{1,\frac{n}{p-1}}$ the identity $\omega= T(d \omega)+d T(\omega)$, has uniformly bounded norm by Proposition $4.1$ in \cite{iwa5} and the fact that $\frac{n}{p-1}$ is bounded away uniformly from $1$ and $+\infty$. Finally, we can define the $(n-1)$-form
\begin{equation}
\omega := \sum_{j=1}^{n} (-1)^{j-1} E^j dx^1 \wedge ... \wedge \widehat{dx^j} \wedge ... \wedge dx^n,
\end{equation}
where $\hat{\cdot}$ means that we are skipping the $j$-component, so that $*d\omega = \Div(E)=0$ and we can verify that $h:=\nabla E_{\frac{n}{n-p+1} } b + d \big( E_{\frac{n}{p-1}} T(\omega) \big)$ is in the Hardy space and satisfies the inequality \eqref{eq.CMLSlocal}.
\end{proof}

We say that a function $u$ has \textbf{bounded mean oscillations}, and we write $u \in BMO(\R^n)$, if
\begin{equation}
[u]_{BMO(\R^n)}:= \sup \fint_B |u(y)-[u]_B| dy < +\infty,
\end{equation}
where the supremum is taken over all balls $B \subset \R^n$.

Restrictions of BMO functions are not of bounded mean oscillations in the corresponding subset, however we can suitably cut-off BMO functions under a stronger Sobolev assumption. This fact, by now classical, can be found for example in the proof of Lemma $1$ in \cite{strz2}.  
\begin{lemma}\label{lemma.cutoff}
Consider a ball $B(x_0,4R_0) \subset \R^n$, a tangent vector field $Y \in W^{1,n}(B_{4R_0}(x_0);T\N)$ and a cut-off function $\eta_{2R_0} \in C^{\infty}_c(B_{4R_0}(x_0))$ such that $\eta_{2R_0} \in [0,1]$, $\eta_{2R_0} \equiv 1$ on $B_{2R_0}(x_0)$ and $|\nabla \eta_{2R_0}| \le 4 R_0^{-1}$. Then the vector field $\Tilde{Y}:= \eta_{2R_0} (Y-[Y]_{B_{4 R_0}(x_0)}) \in W^{1,n}(\R^n;\R^K)$ and we have the BMO-seminorm bound
\begin{equation}
[\Tilde{Y}]_{BMO,\R^n} \le C_4 \norm{\nabla Y}_{n,B_{4 R_0}(x_0)},
\end{equation}
for a dimensional constant $C_4=C_4(n)$.
\end{lemma}
\begin{proof}
The proof follows from a simple application of Poincare's and H\"older's inequalities, combined with the bounds on $\eta_{2R_0}$. Indeed, for any ball $B \subset \R^n$ we compute
\begin{align*}
&\fint_B |\Tilde{Y}-[\Tilde{Y}]_B|\le \Big( \fint_B |\Tilde{Y}-[\Tilde{Y}]_B|^n \Big)^{\frac{1}{n}} \le C(n) \Big(\int_B |\nabla \Tilde{Y}|^n \Big)^{\frac{1}{n}} = C(n) \Big(\int_{B \cap B_{4 R_0}(x_0) } |\nabla \Tilde{Y}|^n \Big)^{\frac{1}{n}} \\
&\le C(n) \Big(\int_{B_{4 R_0}(x_0)} |\nabla Y|^n \Big)^{\frac{1}{n}} + C(n) \Big(\int_{B_{4 R_0}(x_0)} |\nabla \eta_{2R_0}|^n |Y-[Y]_{B_{4 R_0}(x_0)}|^n \Big)^{\frac{1}{n}} \\
&\le C(n) \Big(\int_{B_{4 R_0}(x_0)} |\nabla Y|^n \Big)^{\frac{1}{n}} + C(n) \Big(\int_{B_{4 R_0}(x_0)} (4R_0)^{-n} |Y-[Y]_{B_{4 R_0}(x_0)}|^n \Big)^{\frac{1}{n}}\\
&\le C(n) \Big(\int_{B_{4 R_0}(x_0)} |\nabla Y|^n \Big)^{\frac{1}{n}}. \qedhere
\end{align*}
\end{proof}
A classical result due to Fefferman \cite{fef} says that $BMO(\R^n)$ is the dual space of $\mathcal{H}^1(\R^n)$ and the duality coupling is bounded: for any $v \in BMO(\R^n)$ and $h \in \mathcal{H}^1(\R^n)$ we have
\begin{equation}\label{eq.hardybmo}
\left| \int  v h \right| \le C_5(n) [v]_{BMO(\R^n)} \norm{h}_{\mathcal{H}^1(\R^n)}.
\end{equation}
\section{Main Regularity result}\label{sec.regularity}
Firstly, let us recall the set-up of Theorem \ref{th.SUregularity}. Let $(\N,h)$ be a smooth compact Riemannian homogeneous space with a left invariant metric $h$, identified to a quotient of Lie groups $G/H$, where $G$ is connected and $H$ is a closed subgroup of it. The Lie algebra of $G$ is denoted by $\mathfrak{g}$. We are given Helein fields $(X_1,...,X_\ell)$ and $(Y_1,...,Y_\ell)$ associated to it.

For any $p \in (n,P_0)$, we are given a generalized $p$-harmonic map $u_p \in W^{1,p-1}(\M;\N)$, that is a solution of \eqref{eq.generalizedhomogeneous} which we recall here:
\begin{equation*}
\int_{\M} |\nabla^g u|^{p-2} h(\nabla^g u,X_a(u)) \nabla^g \zeta =0, \quad \forall \zeta \in C^{\infty}_c(\M).
\end{equation*}
Moreover, by Definition \ref{def.helein} we know that
\begin{equation*}
|\nabla^g u|^{p-2} \nabla^g u= \sum_{a=1}^\ell h(|\nabla^g u|^{p-2} \nabla^g u,X_a(u)) Y_a(u),
\end{equation*}
hence we have
\begin{equation}\label{eq.veryweakhomogeneous}
\int_{\M} |\nabla^g u|^{p-2} \nabla^g u \nabla^g \zeta - \sum_{a=1}^\ell |\nabla^g u|^{p-2} h(\nabla^g u,X_a(u)) Y_a(u) \nabla^g \zeta =0, \quad \forall \zeta \in C^{\infty}_c(\M;\R^K).
\end{equation}
Given the local nature of Theorem \ref{th.SUregularity} we will restrict the attention to $\M=B_{4 R_0}(x_0) \subset \R^n$ and drop the index $g$, keeping in mind that all the constants will depend also on $(\M,g)$.
\begin{proof}[Proof of Theorem \ref{th.SUregularity}]
Consider an arbitrary ball $B_{4R}(x) \subset B_{4 R_0}(x_0)$.
Define cut-off functions $\eta_{R} \in C^\infty_c(B_{2 R}(x))$ such that $\eta_{R} \equiv 1$ on $B_{R}(x)$, $\eta_{R} \in [0,1]$ and $|\nabla \eta_{R}| \le 2 R^{-1}$, and $\eta_{2 R} \in C^\infty_c(B_{4 R}(x))$ with analogous properties as in Lemma \ref{lemma.cutoff}. We set
\begin{equation*}
\Tilde{u}:= \eta_{R} (u-[u]_{B_{2 R}(x)}) \in W^{1,n}(\R^n;\R^K).
\end{equation*}
It is clear that $\Tilde{u}$ is supported in the ball $B_{2R}(x)$. If we set $\Tilde{Y}_a:= \eta_{2R} (Y_a(u)-[Y_a(u)]_{B_{4R}})$, we know that $\Tilde{Y}_a=Y_a(u)-[Y_a(u)]_{B_{4R}}$ on $B_{2R}(x)$; moreover, from Lemma \ref{lemma.cutoff} and the smoothness of the fields $Y_a$'s we deduce 
\begin{equation}\label{eq.boundY}
[\Tilde{Y}_a]_{BMO,\R^n} \le C(n) \norm{\nabla [Y_a(u)] }_{n,B_{4R}(x) } = C(n) \norm{(\nabla^{\N} Y_a)(u) \nabla u }_{n,B_{4R}(x) } \le C(n,\N,h) \norm{\nabla u}_{n,B_{4 R}(x)}.
\end{equation}
For $\e=n-p$ we consider the Hodge decomposition constructed in Theorem \ref{th.hodge} with $r=n$
\begin{equation*}
|\nabla \Tilde{u}|^{\e} \nabla \Tilde{u}=: \nabla v + H.
\end{equation*}
From Theorem \ref{th.hodge} and inequalities \eqref{eq.poincare}, \eqref{eq.boundc1} and \eqref{eq.gradienthodge} we deduce the bounds
\begin{equation*}
\norm{\nabla v}_{\frac{n}{n-p+1},\R^n} \le C(n) \norm{\nabla \Tilde{u}}_{n,\R^n}^{1-p+n} \le C(n) \norm{\nabla u}_{n,B_{2R}(x)}^{1-p+n},
\end{equation*}
and
\begin{equation*}
\norm{H}_{\frac{n}{n-p+1},\R^n} \le C_1(n) |p-n| \norm{\nabla \Tilde{u}}_{n,\R^n}^{1-p+n}\le C(n) |p-n| \norm{\nabla u}_{n,B_{2R}(x)}^{1-p+n}.
\end{equation*}
Set $\Tilde{v}:=\eta_{R} (v-[v]_{B(x,2R)}) \in W^{1,n}_0(B(x,2R))$. From the discussion at the beginning of this Section, we can test the identity \eqref{eq.veryweakhomogeneous} with $\Tilde{v}$ to get
\begin{equation}\label{eq.ELtested}
\int_{B_{2R}(x)} |\nabla u|^{p-2}\nabla u \nabla(\eta_{R} (v-[v]_{B(x,2R)})) = \int_{B_{2R}(x)} \sum_{a=1}^\ell |\nabla u|^{p-2} h(\nabla u,X_a(u)) Y_a(u) \nabla(\eta_{R} (v-[v]_{B(x,2R)})).
\end{equation}
We start rewriting the left-hand-side as follows
\begin{align*}
&\int_{B_{2R}(x)} |\nabla u|^{p-2}\nabla u \nabla(\eta_{R} (v-[v]_{B(x,2R)})) = \int_{B_{2R}(x)} \eta_R |\nabla u|^{p-2}\nabla u \nabla v+(v-[v]_{B(x,2R)})|\nabla u|^{p-2}\nabla u \nabla \eta_{R}\\
&=\int_{B_{2R}(x)} |\nabla u|^{p-2}\nabla \Tilde{u} (\nabla v \pm H)-(u-[u]_{B(x,2R)})|\nabla u|^{p-2}\nabla v \nabla \eta_{R}+(v-[v]_{B(x,2R)})|\nabla u|^{p-2}\nabla u \nabla \eta_{R}\\
&=\int_{B_{2R}(x)} |\nabla u|^{p-2}|\nabla \Tilde{u}|^{2+\e}-(u-[u]_{B(x,2R)})|\nabla u|^{p-2}\nabla v \nabla \eta_{R}+(v-[v]_{B(x,2R)})|\nabla u|^{p-2}\nabla u \nabla \eta_{R} -|\nabla u|^{p-2}\nabla \Tilde{u} H.
\end{align*}
Therefore, by the properties of $\eta_R$, $\delta$-weighted Young's inequality with exponents $(\tfrac{n}{p-2},n,\tfrac{n}{n-p+1})$ or $(\tfrac{n}{p-1},\tfrac{n}{n-p+1})$, H\"older's inequality with exponents $(\tfrac{n}{p-1},\tfrac{n}{n-p+1})$, combined with inequalities \eqref{eq.poincare}, \eqref{eq.boundc1} and \eqref{eq.gradienthodge} we obtain a bound for the left-hand-side of \eqref{eq.ELtested}:
\begin{align*}
&\int_{B_{2R}(x)} |\nabla u|^{p-2}\nabla u \nabla(\eta_{R} (v-[v]_{B(x,2R)})) \ge \int_{B_{2 R}(x)} |\nabla u|^{p-2}|\nabla \Tilde{u}|^{2+\e}-\delta \Big[ \big| \tfrac{u-[u]_{B(x,2R)}}{R}\big|^n+|\nabla v|^{\frac{n}{n-p+1}} \Big]\\
&- \int_{A(x,R,2R)} C(\delta,n)|\nabla u|^n - \norm{H}_{\frac{n}{n-p+1},B(x,2R)} \norm{|\nabla u|^{p-2}\nabla \Tilde{u}}_{\frac{n}{p-1},B(x,2R)} \\
&\ge D_n(u;B(x,R))-C(n) \big( \delta +|p-n| \big) D_n(u;B(x,2R))-C(\delta,n) D_n(u;A(x,R,2R)).
\end{align*}
Recall that by definition of generalized $p$-harmonic map, the vector fields $W_a:=|\nabla u|^{p-2} h(\nabla u,X_a(u))$ are divergence free in the distributional sense. Therefore, combining with the properties of $\eta_{2R}$, the right-hand-side of \eqref{eq.ELtested} can be rewritten as
\begin{align*}
&\int_{B_{2R}(x)} \sum_{a=1}^\ell |\nabla u|^{p-2} h(\nabla u,X_a(u)) Y_a(u) \nabla(\eta_{R} (v-[v]_{B(x,2R)}))= \int_{B_{2R}(x)} \sum_{a=1}^\ell W_a (Y_a(u)-[Y_a(u)]) \nabla(\eta_{R} (v-[v]_{B(x,2R)}))\\
&= \int_{B_{2R}(x)} \sum_{a=1}^\ell W_a \eta_{2R} (Y_a(u)-[Y_a(u)]) \nabla(\eta_{R} (v-[v]_{B(x,2R)}))=\int_{B_{2R}(x)} \sum_{a=1}^\ell W_a \Tilde{Y}_a(u) \nabla(\eta_{R} (v-[v]_{B(x,2R)})).
\end{align*}
This can be estimated through Lemma \ref{lemma.cutoff} combined with the inequalities \eqref{eq.boundY}, \eqref{eq.hardybmo} and \eqref{eq.CMLSlocal}:
\begin{align*}
&\int_{B_{2R}(x)} \sum_{a=1}^\ell W_a \Tilde{Y}_a(u) \nabla(\eta_{R} (v-[v]_{B(x,2R)}))\le C_5 \sum_{a=1}^\ell \norm{W_a \nabla( \eta_{R} (v-[v]_{B(x,2R)}))}_{\mathcal{H}^1} [\Tilde{Y}_a(u)]_{BMO}\\
&\le C_3 C_5 \sum_{a=1}^\ell \norm{W_a}_{\frac{n}{p-1},B(x,2R)} \norm{\nabla(\eta_{R} (v-[v]_{B(x,2R)}))}_{\frac{n}{n-p+1},B(x,2R)} [\Tilde{Y}_a(u)]_{BMO} \le C(n,\N,h) D_n(u;B(x,4R))^{\frac{n+1}{n}}.
\end{align*}
Notice that in the last step we have also used that the Helein fields $X_a$'s and $Y_a$'s are smooth on the closed manifold $\N$.
We can combine the two estimates just obtained, worsening the estimate by trivial containments, to deduce from \eqref{eq.ELtested} that
\begin{equation*}
D_n(u;B(x,R)) \le C(n) \big( \delta +|p-n| \big) D_n(u;B(x,4R))+C(\delta,n) D_n(u;A(x,R,4R)) + C(n,\N,h) D_n(u;B(x,4R))^{\frac{n+1}{n}}.
\end{equation*}
By monotonicity of the $n$-energy we get
\begin{equation*}
D_n(u;B(x,4R)) \le D_n(u;B(x_0,4R_0)) \le \e_0^n,
\end{equation*}
hence, together with a standard hole-filling technique, we arrive to
\begin{equation}
D_n(u;B(x,R)) \le \Big[ \tfrac{C(n)( \delta +|p-n|)}{1+C(\delta,n)}+\tfrac{C(\delta, n)}{1+C(\delta,n)}+ C(n,\N,h) \e_0 \Big] D_n(u;B(x,4R))=: \theta D_n(u;B(x,4R));
\end{equation}
Choosing $\delta=\delta(n):=(4 C(n))^{-1}$, $P_0=P_0(n):=n+(4 C(n))^{-1}$, and afterwards
\begin{equation*}
\e_0=\e_0(n,\N,h):= \tfrac{1}{2} \tfrac{1}{2(1+C(\delta(n),n))} C(n,\N,h)^{-1},
\end{equation*}
we can ensure that $\theta=\theta(n)<1$, hence by iterating and interpolating, we can classically derive that $u \in C^{0,\alpha}_{loc}(B(x_0,4R_0))$ for $\alpha=\alpha(n):=-\log_4(\theta)$, along with the bound
\begin{equation*}
[u]_{C^{0,\alpha}(B_{2R_0}(x_0))} \le C(n,\N,h) D_n(u;B_{4R_0}(x_0))^{\frac{1}{n}} \le C(n,\N,h)\e_0.
\end{equation*}
Equivalently, $u \in W^{1,P_0}_{loc}(B(x_0,4R_0))$ for some $P_0=P_0(n) >n$ and
\begin{equation*}
D_{P_0}(u,B_{2R_0}(x_0))^{\frac{1}{P_0}} \le C(n,\N,h) D_n(u;B_{4R_0}(x_0))^{\frac{1}{n}} \le C(n,\N,h) \e_0.
\end{equation*}
For all $p \in (n,P_0)$, we can now rewrite \eqref{eq.veryweakhomogeneous} for all test maps $\zeta \in C^{\infty}_c (B_{2 R_0}(x_0);\R^K)$
\begin{align*}
\int_{B_{2 R_0}(x_0)} |\nabla u|^{p-2} \nabla u \nabla \zeta &= \int_{B_{2 R_0}(x_0)} \sum_{a=1}^\ell |\nabla u|^{p-2} h(\nabla u,X_a(u)) Y_a(u) \nabla \zeta \\
&= \int_{B_{2 R_0}(x_0)} \sum_{a=1}^\ell |\nabla u|^{p-2} h(\nabla u,X_a(u)) \nabla( Y_a(u) ) \zeta,
\end{align*}
which is \eqref{eq.weakhomogeneous}, to which we can apply Theorem $2$ in \cite{tor} (by possibly further decreasing the value of $\e_0$) to conclude.
\end{proof}
A standard covering argument yields the following Corollary.
\begin{corollary}\label{cor.global}
Let $P_0:=P_0(n,\M,g,\N,h) >n$ and $\alpha_0:=\alpha_0(n,\M,g,\N,h) \in (0,1)$ be the exponents found in Theorem \ref{th.SUregularity}. If $u \in W^{1,n}(\M;\N)$ is a generalized $p$-harmonic map for $p \in (n,P_0)$, then $u \in W^{1,P_0}(\M;\N)$ solves \eqref{eq.weaklypharmonic}, and $u \in C^{1,\alpha_0}(\M;\N)$.
\end{corollary}
Under a uniform $D_n$-energy bound, we can prove this regularity to be uniform away from finitely many points, where bubbles are forming, as formalized in our previous paper \cite{dim}.

The same method of proof can be adapted to show the following Theorem.
\begin{theorem}\label{th.highsubcritical}
For any $\underline{P}>n$, there exist constants $\overline{P}:=\overline{P}(n,\underline{P},\M,g,\N,h) \in (\underline{P},\underline{P}+1)$ and $\alpha_1:=\alpha_1(n,\underline{P},\M,g,\N,h), \ \e_1:=\e_1(n,\underline{P},\M,g,\N,h) \in (0,1)$, such that the following statement holds. Assume that $(u_p)_{(\underline{P},\overline{P})}\in W^{1,\underline{P}}(B(x_0,4R_0);\N)$ is a family of generalized $p$-harmonic maps defined on a ball $B(x_0,4R_0) \subset \M$, satisfying uniformly the bound $ R_0^{\underline{P}-n} D_{\underline{P} }(u_p,B(x_0,4R_0)) \le \e_0^{\underline{P} }$. Then $u_p \in W^{1,\overline{P}}(B(x_0,2R_0);\N)$ solve \eqref{eq.weaklypharmonic}, and $u_p \in C^{1,\alpha_1}(B(x_0,R_0);\N)$ for all $p \in (\underline{P},\overline{P})$, along with uniform bounds.
\end{theorem}
Observe that, even though $\underline{P}>n$ implies a-priori the applicability of Sobolev's embedding, this is not enough to deduce a Caccioppoli-type inequality and bootstrap to the higher order $C^{1,\alpha}$-regularity as the maps $(u_p)$ do not belong to the energy space compare with Colombo-Tione's result from \cite{col} already recalled in the Introduction.

\end{document}